\documentclass[hidelinks]{amsart}
\usepackage{articletemplate_arxiv, alphabeta}
\usepackage{needspace}

\begin{document}
\title[Numbers with at most $2$ prime factors]{Numbers with at most $2$ prime factors in short arithmetic progressions}
\author{Mayank Pandey}

\begin{abstract}
    We show that if $\frac{L}{\varphi(q)\log X}\to\infty$ as $X\to\infty$, almost all $(a, x)\in (\ZZ/q\ZZ)^\times\times [X, 2X]$ 
    are such that there exists a product of at most two primes in $[x, x + L]$ congruent to $a~(\text{mod }{q})$. 
\end{abstract}

\maketitle
\section{Introduction and statement of results}
Based on probabilistic considerations (a standard analysis of the so-called ``coupon collector problem''), 
one might conjecture the following:

\begin{conjecture}\label{conj:main_prime}
    Take $X$ large and $L\ll X, q\ge 1$. Furthermore, writing
    \[
        A = \frac{L}{\varphi(q)\log X},
    \]
    suppose that $A\gg 1$. 
    Then, for all but a $o_{A\to\infty}(1)$ proportion of $(a, x)\in (\ZZ/q\ZZ)^\times\times [X, 2X]$, there
    exists a prime $p\equiv a(q)$ in $[x, x + L]$. 
\end{conjecture}
In the case $q = 1$, this was shown by Heath-Brown~\cite{He} assuming the Riemann hypothesis and certain hypotheses on the
distribution of gaps between zeros of the zeta function. Selberg~\cite{Se} and Tur\'an~\cite{Tu} (resp.) have shown Conjecture~\ref{conj:main_prime} as long as
$A$ grows slightly more quickly than $\log^2 X$ assuming the Riemann hypothesis and
Generalized Riemann hypothesis in the cases $q = 1$ and $L\asymp X$ respectively.

Currently, unconditional results toward Conjecture~\ref{conj:main_prime} are quite poor compared to 
conjectures. In the case $q = 1$, the current best result is due to Jia~\cite{Ji} with intervals
of length $X\efr{1}{20}$.

A weakening of the conjecture that has been studied with more success is that in which one asks 
for numbers with at most $2$ prime factors instead of primes (which for the remainder of the paper
we refer to as almost primes).

In 1989, Mikawa~\cite{Mi} showed that if $A/\log^5X\to\infty$ and $L\asymp X$, then the 
conjecture holds for almost primes: almost all reduced residue classes modulo $q$ contain an 
almost prime $\asymp X$.

A more recent result of Matom\"aki~\cite{Ma} gives the conjecture for almost primes in short intervals by showing the 
following:
\begin{theorem}[{\cite[Theorem 1.1]{Ma}}]\label{thm:short_int}
    Let $c > 0$ be sufficiently small. Then, for $2\le h\le X^{\frac{1}{100}}$, we have that
    \[
        \sum_{x - h\log X < n\le x}\charf{\Omega(n)\le 2}\ge ch
    \]
    for all $x\in [X, 2X]$ outside of a set of measure $O(X/h)$.
\end{theorem}
Our main result is the following generalization of Matom\"aki and Mikawa's results, and at 
the same time an improvement over Mikawa's result when specialized to the case of $L\asymp X$.
\begin{theorem}\label{thm:main}
    Let $c > 0$ be a sufficiently small constant and $L\ge 2, q\ge 1, X$ large.
    Writing
    \[
        A = \frac{L}{\varphi(q)\log X},
    \]
    suppose we have that 
    \[
        1\ll A\le X^c.
    \]
    For $(a, x)\in(\ZZ/q\ZZ)^\times\times [X, 2X]$, let
    \[
        \mc E = \bigg\{(a, x)\in (\ZZ/q\ZZ)^\times\times[X, 2X] : \bigg|\ssum{x - L < n\le x\\ n\equiv a(q)\\ p | n\implies p > X^{\frac{1}{8}}}\charf{\Omega(n)\le 2}\bigg|\le cA \bigg\}.
    \]
    Then, the measure of $\mc E$ is $\ll\frac{\varphi(q)X}{A}$.
\end{theorem}
We state a corollary of the main theorem that may be of interest.
Let $\psi:\ZZ_{> 0}\to\RR_{> 0}$ be such that $\psi(q)\to\infty$ as $q\to\infty$.
Then, almost all residue classes modulo $q$ contain some $q^{\frac 18}$-rough\footnote{By a $z$-rough number, we mean one with no prime factors $\le z$} $n\le \psi(q)\varphi(q)\log q$ 
with $\Omega(n)\le 2$.

We remark that by the prime number theorem, $A$ is, up to a constant, the average 
number of $n\equiv a(q)$ in $[x - L, x]$ with $\Omega(n)\le 2, p |n\implies p > X^{\frac 18}$
as $(a, x)$ ranges over $(\ZZ/q\ZZ)^\times\times [X, 2X]$. In particular, our result can clearly not 
be improved when $A\ll 1$ (in which case the trivial bound of $\varphi(q)X$ matches the main 
theorem). On the other hand, our result is nontrivial if $A\to\infty$ as $X\to\infty$. In this 
sense, our result is optimal, as was Matom\"aki's~\cite{Ma} (after restricting to $X^{\frac 18}$-rough
numbers).


Our methods are a combination of the those of Matom\"aki and Mikawa, utilizing the more efficient 
treatment of sieve weights of Matom\"aki, and using more extensively the well-factorability of 
sieve weights to use the Weil bound as Mikawa does in place of the 
$\GL(2)$ spectral theory done by Matom\"aki. Such a path to Matom\"aki's result is referred to in the remarks
of~\cite{Ma}, coming from a comment of Granville.

%
%

We remark that it should be possible using the ingredients of this paper to improve a result of Mikawa~\cite{Mi2}, and show a bound of 
\[
    \sum_{\substack{p\le X\\ p\equiv a(q)}} 1\le (18 + o(1))\frac{X}{\varphi(q)\log \frac{X^{6}}{q}}.
\]
for almost all $a\in(\ZZ/q\ZZ)^\times$ for $X^{\frac 67}\le \varphi(q)\ll o(X/\log X)$. We thank K. Matom\"aki for pointing this out, and we leave the details to the interested reader.

\subsection{Notation and conventions}
Throughout, we take $X$ some large parameter, and let $L, q$ be as in the statement of the
main theorem.

We use Vinogradov notation $\ll, \gg, \asymp, O(-)$ as usual, with any dependencies of the implied 
constants indicated in subscripts. When we 
refer to $o(Y)$, we mean a quantity that is bounded in magnitude by $C(X)Y$ for some 
$C(X)$ going to zero as $X\to\infty$. 

For $\mc A\subset\NN$, we write 
\[
    S(\mc A, z) = \sum_{n\in\mc A}\charf{p | n\implies p > z}.
\]
For $(a, x)\in (\ZZ/q\ZZ)^\times\times [X, 2X]$, we let
\[
    \mc A(x, a) = \set{x < n\le x + L : n\equiv a(q)}, 
    \mc A_d(x, a) = \set{n\in\mc A(x, a) : d | n}, 
\]
and write $A_d(x, a) = \#\mc A_d(x, a)$.

For some $a\mod d$ with $(a, d) = 1$, we let $\conj a$ be so that $\conj aa\equiv 1(d)$. Which 
modulus $d$ is to be taken will be clear from context or otherwise specified.

\section{Standard lemmas}
We use at a few points the following basic estimates. These are standard, though we know not a reference for some of them, so we provide quick proofs.
\begin{lemma}\label{lem:gcd_sum}
    For all $q, X\ge 1$, we have
    \[
        \sum_{n\le X} (q, n)\le d(q)X.
    \]
\end{lemma}
\begin{proof}
    We have that 
    \[
        \sum_{n\le X} (q, n) = \sum_{n\le X}\ssum{d | (n, q)}\varphi(d) = \sum_{d | q} \varphi(d) \ssum{n\le X\\ d | n } 1\le d(q) X,
    \]
    as desired.
\end{proof}
\begin{lemma}\label{lem:coprime_sum}
    For $X, q\ge 1$, we have
    \[
        \ssum{n\le X\\(n, q) = 1} 1 = \frac{\varphi(q)}{q}X + O(d(q)).
    \]
\end{lemma}
\begin{proof}
    By M\"obius inversion, we have that 
    \[
        \ssum{n\le X\\ (n, q) = 1} 1 = \sum_{d | q}\mu(d)\ssum{n\le X\\ d|n} 1 = X\sum_{d|q}\frac{\mu(d)}{d} + O(d(q)).
    \]
    The desired result follows upon noting that 
    \[
        \frac{\varphi(q)}{q} = \sum_{d | q}\frac{\mu(d)}{d}.
    \]
\end{proof}

\subsection{Sums of incomplete Kloosterman sums}

We shall use the following standard estimate for incomplete Kloosterman sums.

\begin{lemma}\label{lem:kloost_est}
    We have that for all $d\ge 1, \eps > 0$
    \[
        \sum_{\substack{n\le X\\ (n, qd) = 1}} e\pfrc{a\conj n}{q}\ll_\eps 
        (qd)^\eps(a, q)^{\frac{1}{2}}q^{\frac{1}{2} + \eps}\bigg(1 + \frac{X}{q}\bigg).
    \]
\end{lemma}
\begin{proof}
    This is standard; see for example~\cite[Lemma 3]{Mi}.
\end{proof}

\section{Setup}
\subsection{The weighted sieve}
We shall be following the setup of~\cite[\S2]{Ma} in using Richert's weighted sieve. 
So that we may directly make use 
of estimates such as~\cite[(49)]{Ma}, we shall use the same sieve weights as in~\cite{Ma}, with only slight modifications. 
Take $\eps > 0$ small, and write
$\kappa = \frac{17}{31}$. Then, take
\begin{align*}
    D = X^{\kappa - 100\eps}, z = D^{\frac{1}{4}}, y = X^{\frac{1}{2} - 10\eps}.
\end{align*}
Define the coefficients
\[
    w_n = 1 - \frac{1}{\eta}\ssum{p | n\\ z\le p < y}\bigg(1 - \frac{\log p}{\log y}\bigg).
\]
where $\eta = \frac{1 - 60\eps}{1 - 20\eps}$.
The key property we require of $w_n$ is the following:
\begin{lemma}\label{lem:weighted_sieve_upper}
    For all but a measure $\ll\frac{LX}{z}$ subset of $(a, x)\in(\ZZ/q\ZZ)^\times\times [X, 2X]$, we have that
    \[
        \ssum{x - L < n\le x\\ n\equiv a(q)\\p | n\implies p > z}\charf{\Omega(n)\le 2}\ge\frac{2 - 80\eps - 800\eps^2}{1 - 40\eps + 400\eps^2}\ssum{n\in\mc A(x, a)\\ p | n\implies p > z} w_n.
    \]
\end{lemma}
\begin{proof}
    Outside of a measure $\frac{LX}{z}$ subset of $(a, x)\in (\ZZ/q\ZZ)^\times\times [X, 2X]$, there does not exist $n\in\mc A(x, a)$ such that $p^2 | n$ for some $p > z$. For such $(a, x)$, we have that
    \[
        \ssum{n\in\mc A(x, a)\\p | n\implies p > z}\charf{\Omega(n)\le 2} = \ssum{n\in\mc A(x, a)\\p | n\implies p > z}\charf{\omega(n)\le 2}.
    \]
    Thus, it suffices to show that
    \[
        \ssum{n\in\mc A(x, a)\\p | n\implies p > z}\charf{\omega(n)\le 2}\ge \bigg(\frac 12 - \delta\bigg)\ssum{n\in\mc A(x, a)\\ p | n\implies p > z} w_n.
    \]
    To see this, note that 
    \begin{align*}
        \ssum{n\in\mc A(x, a)\\ p | n\implies p > z} w_n&\le\ssum{n\in\mc A(x, a)\\ p | n\implies p > z}\bigg(-\frac{1}{\eta}\omega(n) + 1 + \frac{1}{\eta}\frac{1}{\log y}\sum_{p | n}\log p\bigg)\\
        &\le\ssum{n\in\mc A(x, a)\\ p | n\implies p > z}\bigg(-\frac{1}{\eta}\omega(n) + 1 + \frac{1}{\eta}\frac{\log X}{\log y}\bigg)\\
        &\le\ssum{n\in\mc A(x, a)\\ p | n\implies p > z}\bigg(-\frac{1}{\eta}\omega(n) + 1 + \frac{1}{\eta}\frac{2}{1 - 20\eps}\bigg)\\
        &\le\frac{2 - 80\eps - 800\eps^2}{1 - 40\eps + 400\eps^2}\ssum{n\in\mc A(x, a)\\ p | n\implies p > z}\charf{\omega(n)\le 2}.
    \end{align*}
    The desired result follows.
\end{proof}
Fix for now some choice of $(a, x)\in (\ZZ/q\ZZ)^\times\times [X, 2X]$.  By the definition of $w_n$, we have
\begin{equation}\label{eq:wn_expand}
    \ssum{n\in\mc A(x, a)\\p|n\implies p > z} w_n = S(\mc A(x, a), z) - \frac{1}\eta\sum_{w\le p < y} \frac{\log p}{\log y}S(\mc A(x, a)_p, z).
\end{equation}
Take $\beta = 30$, and $w = X^{\eps^2}, E = X^{\eps^3}$. Then, take 
\begin{align*}
    \mc D^{\pm} &= \set{p_1\dots p_r : z\ge p_1 > \cdots > p_r\ge w, p_1\dots p_m p_m^2 < D\forall (-1)^{m + 1} = \pm 1},\\
    \mc E^{\pm} &= \set{p_1\dots p_r : w\ge p_1 > \cdots > p_r, p_1\dots p_m p_m^\beta < E\forall (-1)^{m + 1} = \pm 1},
\end{align*}
and define upper and lower bound linear and $\beta$-sieve weights $\lambda_d^{\pm} = \mu(d)\charf{d\in\mc D^\pm}$ and $\rho_e^{\pm} = \mu(e)\charf{e\in\mc E^\pm}$ respectively. 
Define $\alpha_k^-$ as in~\cite[(15)]{Ma} so that
\[
    \charf{p | n\implies p > z}\ge\sum_{d | n}\alpha_d^-.
\]
Specifically, take $\alpha_d^- = \lambda_d^+\rho_d^- + \lambda_d^-\rho_d^+ - \lambda_d^+\rho_d^+$.

We now construct upper bound sieve weights for $S(\mc A(x, a)_p, z)$. For any scale $P$, let
\begin{align*}
    \mc D^{\pm}_P &= \set{p_1\dots p_r : z\ge p_1 > \cdots > p_r\ge w, p_1\dots p_m p_m^2 < D/P\forall (-1)^{m + 1} = \pm 1},\\
\end{align*}
Now, for $w\le P < z$, define upper and lower bound sieve weights $\lambda_{d,P}^{\pm} = \mu(d)\charf{d\in\mc D^\pm}$, and
let $\alpha_d^+ = \lambda_{d, P}^+\rho_d^+$, so that 
\[
    \charf{p | n\implies p > z}\le\sum_{d | n}\alpha_{d, P}^+.
\]
Putting this all together with (\ref{eq:wn_expand}), we obtain that  
\begin{align*}
    \ssum{n\in\mc A(x, a)\\ p|n\implies p > z} w_n&\ge \sum_{(d, q) = 1} \alpha_d^-A_d(x, a) - \sum_{z\le P < y}\ssum{p\sim P\\ p\nmid q}\sum_{(d, q) = 1}\alpha_{d, P}^+ A_{dp}(x, a)\\
    &= \sum_{(d, q) = 1} (\alpha_{d}^- - \beta_d) A_d(x, a),
\end{align*}
where 
\[
    \beta_{d} = \sum_{z\le P < y}\ssum{ep = d\\ p\sim P} \alpha_{d, P}^+.
\]
Here $P$ runs over powers of $2$, and we have made use of the fact that $\mc A(x, a)$ only contains integers relatively prime to $q$.
Rearranging, it follows that
\[
    \ssum{n\in\mc A(x, a)\\ p|n\implies p > z} w_n\ge \frac{L}{q}M(z, y) + E(x, a, y, z),
\]
where 
\begin{align*}
    M(z, y) &= \sum_{(d, q) = 1}\frac{\alpha_d^- - \beta_{d}}{d},
    E(x, a, y, z) &= \sum_{(d, q) = 1}(\alpha_d^- - \beta_d)\bigg(A_d(x, a) - \frac{L}{qd}\bigg).
\end{align*}
The main theorem reduces to showing the following two estimates:
\begin{equation}\label{eq:sieve_main_lower_bd}
    M(z, y)\gg \frac{q}{\varphi(q)}\frac{1}{\log X},
\end{equation}
\begin{equation}\label{eq:remainder_averaged}
    \sumCp_{a(q)}\int_X^{2X} |E(x, a, y, z)|^2dx\ll A \varphi(q)X = \frac{LX}{\log X}.
\end{equation}
We end this section by showing (\ref{eq:sieve_main_lower_bd}) with a computation along the lines of \S3 of~\cite{Ma}. 
Write 
\[
    V(w) = \prod_{\substack{p\le w\\ p\nmid q}}\bigg(1 - \frac{1}{p}\bigg), V(w, z) = \prod_{\substack{w < p\le z\\ p\nmid q}}\bigg(1 - \frac{1}{p}\bigg), 
    V(z) = V(w)V(w, z).
\]
It is clear that $V(z)\asymp\frac{q}{\varphi(q)}\frac{1}{\log X}$. 
Then, by the fundamental lemma of the sieve (see~\cite[Lemma 6.8]{FI}, for example), we have that 
\[
    \ssum{(e, q) = 1}\frac{\rho_e^\pm}{e} = (1 + O(\eps))V(w).
\]
We also have that with $f, F$ the linear sieve functions (see \S3 of~\cite{Ma} or Chapter 12 of~\cite{FI})
\begin{align*} 
    \sum_{(d, q) = 1}\frac{\lambda_d^+}{d}&\le \bigg(F\pfrc{\log D}{\log z} + O(\eps)\bigg)V(w, z),\\
    \sum_{(d, q) = 1}\frac{\lambda_d^-}{d}&\ge \bigg(f\pfrc{\log D}{\log z} + O(\eps)\bigg)V(w, z),\\
    \sum_{(d, q) = 1}\frac{\lambda_{d,P}^+}{d}&\le\bigg(F\pfrc{\log (D/P)}{\log z} + O(\eps)\bigg)V(w, z).
\end{align*}
Now, we have that by the definition of $\alpha_d^-, \beta_d$,
\begin{align*}
    M(z, y)&\ge V(z)\bigg(f(4) - \sum_{z < p\le y}\bigg(1 - \frac{\log p}{\log y}\bigg)F\pfrc{\log(D/p)}{\log z} + O(\eps)\bigg)\\
    &\ge V(z)\bigg(\frac{1}{2}e^\gamma\log 3 - 2e^\gamma\int_{1/4}^{2\gamma}(1 - 2\gamma\alpha)\frac{d\alpha}{4\alpha(1 - \alpha)} + O(\eps)\bigg).
\end{align*}
With our choice of $\kappa = \frac{17}{31}$, this is 
\[
    \frac{1}{2}e^\gamma\log 3 - 2e^\gamma\int_{1/4}^{2\kappa}(1 - 2\kappa\alpha)\frac{d\alpha}{4\alpha(1 - \alpha)}\ge 0.0166,
\]
so (\ref{eq:sieve_main_lower_bd}) follows.

It therefore remains to show (\ref{eq:remainder_averaged}), which we do in the following two sections.

\section{Type I sums on average}
We shall use the following generalization of Proposition 5.1 in~\cite{Ma} to bound the left-hand side of (\ref{eq:remainder_averaged}).
\begin{proposition}\label{prop:typeI_average}
    Take $(a_d)_{d\le D_0}$ a sequence supported on $d\le D_0$ coprime to $q$ for some $D_0\le X^{1 - \delta}$ for some absolute $\delta > 0$.
    Suppose also that $L\le \frac{1}{3}X$.
    Then, for any smooth function $g$ compactly supported on $[1/2, 3]$, we have that 
    \begin{multline*}
        \int_\RR \sumCp_{a(q)} g\pfrc{x}{X}\bigg(\ssum{d, m\\ x - L < md\le x\\ md\equiv a(q)}a_d - \frac{L}{q}\sum_{d}\frac{a_d}{d}\bigg)^2dx\\
        = S_1 + S_2 + S_3 + O\bigg(1 + \frac{L^2d(q)}{q}\bigg|\sum_d\frac{a_d}{d}\bigg|
        \bigg|\sum_d a_d\bigg|\bigg),
    \end{multline*}
    where 
    \begin{align*}
        S_1 &= 2\hat g(0)X\varphi(q)\sum_d\gamma_{d,L/q}\bigg(\sum_{m\equiv 0(d)}\frac{a_m}{m}\bigg),\\
        S_2 &= \sum_{0 < |k|\le L/q}\ssum{d_1, d_2\le D_0\\ (d_1, d_2) | qk}a_{d_1}a_{d_2}\int_{\max(0, qk)}^{L + \min(0, qk)}\\ 
        &\hspace{1cm}\bigg(\ssum{m_1, m_2\\ d_1m_1 = d_2m_2 + kq\\ (m_1, q) = 1} g\bigg(\frac{d_1m_1}{X} + \frac{y}{X}\bigg)
        - \frac{1}{[d_1, d_2]}\cdot\frac{\varphi(q)}{q}\hat g(0)X\bigg)dy,\\
        S_3 &= \frac{\varphi(q)}{q}\int_0^L\sum_{(n, q) = 1}g\bigg(\frac{n}{X} + \frac{y}{X}\bigg)\bigg(\sum_{d | n} a_d\bigg)^2 dy \\ 
        &\hspace{1cm} - \frac{\varphi(q)}{q}\hat g(0)LX\frac{1}{X^{10}}
        \ssum{n\le X^{10}\\ (n, q) = 1}\bigg(\sum_{d | n} a_d\bigg)^2,
    \end{align*}
    with
    \[
        \gamma_{d,L/q} = \ssum{m\ge 1\\ (m, d) = 1}\bigg(\frac{d}{\pi m}\sin\pfrc{\pi mL}{qd}\bigg)^2.
    \]
\end{proposition}
\begin{proof}
    Expanding, we have that 
    \begin{align}
        S = \int_\RR \sumCp_{a(q)} g\pfrc{x}{X}\bigg(\ssum{d\le D_0, m\\ x - L < md\le x\\ md\equiv a(q)}a_d - \frac{L}{q}\sum_{d\le D_0}\frac{a_d}{d}\bigg)^2dx
        = \Sigma_1 - 2\Sigma_2 + \Sigma_3
    \end{align}
where 
\begin{align*}
    \Sigma_1 &= \int_\RR\sumCp_{a(q)} g\pfrc{x}{X}\bigg(\ssum{d, m\\ x - L < md\le x\\ md\equiv a(q)} a_d\bigg)^2dx,\\
    \Sigma_2 &= \frac{L}{q}\sum_{d_1}\frac{a_{d_1}}{d_1}\int_\RR\sumCp_{a(q)} g\pfrc{x}{X}\bigg(\ssum{d_2, m\\ x - L < md_2\le x\\ md_2\equiv a(q)} a_{d_2}\bigg)dx,\\
    \Sigma_3 &= \frac{L^2}{q^2}\int_\RR\sumCp_{a(q)} g\pfrc{x}{X}dx\bigg(\sum_d\frac{a_d}{d}\bigg)^2.
\end{align*}
We have that 
\begin{multline*}
    \int_\RR\sumCp_{a(q)} g\pfrc{x}{X}\bigg(\ssum{d_2, m\\ x - L < md_2\le x\\ md_2\equiv a(q)} a_{d_2}\bigg)dx \\
    = \sum_{d_2} a_{d_2}\sum_m\sumCp_{\substack{a(q)\\ md_2\equiv a(q)}}\int_{d_2m}^{d_2m + L} g\pfrc{d_2m + y}{X} dy.
\end{multline*}
Then, noting that the sum over $m, a$ may be replaced by just a sum over $(m, q) = 1$, by Lemma~\ref{lem:coprime_sum} and partial summation, we have 
\begin{align*}
    \Sigma_2 &= \frac{L}{q}\bigg(\sum_{d_1}\frac{a_{d_1}}{d_1}\bigg)\sum_{d_2} a_{d_2}\int_{d_2m}^{d_2m + L} \sum_{(m, q) = 1} g\pfrc{d_2m + y}{X} dy \\
    &= \frac{XL^2\varphi(q)}{q^2}\hat g(0)\bigg(\sum_{d}\frac{a_d}{d}\bigg)^2 + O\bigg(\frac{L^2d(q)}{q}\bigg|\sum_d\frac{a_d}{d}\bigg|\bigg|\sum a_d\bigg|\bigg).
\end{align*}
We also have that
\[
    \Sigma_3 = \frac{XL^2\varphi(q)}{q^2}\hat g(0)\bigg(\sum_d\frac{a_d}{d}\bigg)^2.
\]
It follows that
\[
    S = \Sigma_1 - \frac{XL^2\varphi(q)}{q^2}\hat g(0)\bigg(\sum_d\frac{a_d}{d}\bigg)^2 
    + O\bigg(\frac{L^2d(q)}{q}\bigg|\sum_d\frac{a_d}{d}\bigg|\bigg|\sum a_d\bigg|\bigg).
\]
It remains to evaluate $\Sigma_1$. To that end, since only the terms with $(d, q) = 1$ contribute, we have
\begin{align*}
    \Sigma_1 &= \ssum{d_1, d_2, m_1, m_2\\ d_1m_1\equiv d_2m_2(q) \\ |d_1m_1 - d_2m_2|\le L}a_{d_1}a_{d_2}\int_\RR\sumCp_{a(q)} g\pfrc{x}{X}\charf{\substack{x - L < d_1m_1, d_2m_2\le x\\ m_1d_1\equiv m_2d_2\equiv a(q)}}dx\\
    &= \ssum{|k|\le L/q}\int_{\max(0, kq)}^{L + \min(0, kq)}\ssum{d_1, d_2\\ (d_1, d_2) | k}a_{d_1}a_{d_2}\ssum{m_1, m_2\\ d_1m_1 = d_2m_2 + kq\\ (m_1, q) = 1}g\bigg(\frac{d_1m_1}{X} + \frac{y}{X}\bigg) dy.
\end{align*}
Separating out the $k = 0$ contribution, and inserting the definition of $R_{d_1, d_2}$, we obtain that 
\begin{align*}
    \Sigma_1  &= S_2 + X\hat g(0)\frac{\varphi(q)}{q}\ssum{|k|\le L/q}(L - q|k|)\ssum{d_1, d_2\\ (d_1, d_2) | k}\frac{a_{d_1}a_{d_2}}{[d_1, d_2]}\\
    &\vspace{10cm} + \int_0^L\ssum{d_1, d_2} a_{d_1}a_{d_2}\bigg(\ssum{m_1, m_2\\ d_1m_1 = d_2m_2\\ (m_1, q) = 1} g\left(\frac{d_1m_1}{X} + \frac{y}{X}\right) - \frac{1}{[d_1, d_2]}\frac{\varphi(q)}{q}\hat g(0)X\bigg) dy\\
    &=  S_2 + \hat g(0)X\frac{\varphi(q)}{q}\ssum{|k|\le L/q} (L - q|k|)\ssum{d_1, d_2\\ (d_1, d_2) | k}\frac{a_{d_1}a_{d_2}}{[d_1, d_2]} \\
    &+ \int_0^L\sum_{(n, q) = 1} g\bigg(\frac{n}{X} + \frac{y}{X}\bigg)\bigg(\sum_{d | n} a_d\bigg)^2dy 
    - LX\hat g(0)\frac{\varphi(q)}{q}\sum_{d_1, d_2} \frac{a_{d_1}a_{d_2}}{[d_1, d_2]}\\
    &= S_2 + \hat g(0)X\frac{\varphi(q)}{q}\ssum{|k|\le L/q} (L - q|k|)\ssum{d_1, d_2\\ (d_1, d_2) | k}\frac{a_{d_1}a_{d_2}}{[d_1, d_2]} \\
    &\hspace{1cm}+ \int_0^L\sum_{(n, q) = 1} g\bigg(\frac{n}{X} + \frac{y}{X}\bigg)\bigg(\sum_{d | n} a_d\bigg)^2dy \\
    &\hspace{2cm}- LX\hat g(0)\sum_{d_1, d_2} a_{d_1}a_{d_2}\frac{1}{X^{10}}\ssum{n\le X^{10}\\ [d_1, d_2] | n\\ (n, q) = 1} 1 + O(1) \\
    &= S_2 + S_3 + \hat g(0)X\ssum{|k|\le L/q} (L - q|k|)\ssum{d_1, d_2\\ (d_1, d_2) | k}\frac{a_{d_1}a_{d_2}}{[d_1, d_2]} + O(1).
\end{align*}
It therefore remains to show that $S_1$ equals
\[
    S_1' = \hat g(0)X\frac{\varphi(q)}{q}\ssum{|k|\le L/q} (L - q|k|)\ssum{d_1, d_2\\ (d_1, d_2) | kq}\frac{a_{d_1}a_{d_2}}{[d_1, d_2]} - \frac{XL^2\varphi(q)}{q^2}
    \hat g(0)\bigg(\sum_{d | n}  a_d\bigg)^2.
\]
Note that only the terms $(d_1d_2, q) = 1$ contribute, the condition $(d_1, d_2) | kq$ may be replaced by the condition $(d_1, d_2) | k$. Applying this and rearranging, we obtain that 
\[
    S_1' = \hat g(0)Xq\sum_{d_1, d_2}\frac{a_{d_1}a_{d_2}}{d_1d_2}\bigg((d_1, d_2)\ssum{|k|\le L/q\\ (d_1, d_2) | k} \bigg(\frac{L}{q} - |k|\bigg) - \pfrc{L}{q}^2\bigg).
\]
At this point, we are done by the computation evaluating (38) in the proof of Proposition 5.1 of~\cite{Ma}
\end{proof}
Take $g$ some smooth compactly supported function such that $\charf{[1,2]}\le g\le\charf{[1/2, 3]}$. Then, by Proposition~\ref{prop:typeI_average}, we have that
\begin{align*}
    \sumCp_{a(q)}\int_X^{2X}& |E(x, a, y, z)|^2\le\sumCp_{a(q)}\int g\pfrc{x}{X}|E(x, a, y, z)|^2\\
    &\ll |S_1^\alpha| + |S_1^\beta| + |S_2| + |S_3^\alpha| + |S_4^\beta| + O\pfrc{L^2D\log Dd(q)}{q},
\end{align*}
where 
\[
    S_1^\alpha = X\varphi(q)\sum_{d}\gamma_{d, L/q}\bigg(\ssum{m\equiv 0(d)\\ (m, q) = 1} \frac{\alpha_d^-}{d}\bigg),
\]
\begin{multline*}
    S_2 = \sum_{0 < |k|\le L/q} \int_{\max(0, qk)}^{L + \min(0, qk)}\ssum{d_1, d_2\le D\\ (d_1, d_2) | k}\gamma_{d_1}\gamma_{d_2}\\
    \bigg(\ssum{m_1,m_2\\ d_1m_1 = d_2m_2 + k\\ (m_1, q) = 1}g\bigg(\frac{d_1m_1}{X} + \frac{y}{X}\bigg) - \frac{\varphi(q)}{q}\frac{1}{[d_1, d_2]}X\bigg)dy,
\end{multline*}
\begin{multline*}
    S_3^\alpha = \frac{\varphi(q)}{q}\int_0^L\sum_{(n, q) = 1}g\bigg(\frac{n}{X} + \frac{y}{X}\bigg)\bigg(\sum_{d | n} \alpha_d^-\bigg)^2\\
    -\frac{\varphi(q)}{q}\hat g(0) LX\frac{1}{X^{10}}\ssum{n\le X^{10}\\ (n, q) = 1} \bigg(\sum_{d|n}\alpha_d^-\bigg)^2,
\end{multline*}
and $S_1^\beta, S_3^\beta$ are the same as $S_1^\alpha,  S_3^\alpha$ with $\beta_d$ in place of $\alpha_d^-$. Here, $\gamma_d = \alpha_d^- - \beta_d$.
\section{Bounding (\ref{eq:remainder_averaged})}
In this section, we show (\ref{eq:remainder_averaged}) by bounding 
$S_1^\alpha, S_1^\beta, S_2, S_3^\alpha, S_3^\beta$ in the following few subsections.

\subsection{Bounding \texorpdfstring{$S_1^\alpha, S_1^\beta$}{S1}}
We shall show in this section that 
\begin{equation}\label{eq:S1_bd}
    S_1^\alpha, S_1^\beta\ll X\varphi(q)A = \frac{XL}{\log X}.
\end{equation}
Noting that $\gamma_{d, L/q}\ll dL/q$, (\ref{eq:S1_bd}) follows if we can show that 
\begin{equation}\label{eq:S1_desired_est1}
    \sum_d d\bigg(\ssum{(m, q) = 1\\ m\equiv 0(d)}\frac{\alpha_d^-}{d}\bigg)^2, \sum_d d\bigg(\ssum{(m, q) = 1\\ m\equiv 0(d)}\frac{\beta_d}{d}\bigg)^2
    \ll\frac{q}{\varphi(q)}\frac{1}{\log X}.
\end{equation}
Identically to the reduction of (49), (50) to (55) in~\cite{Ma}, we can show that (\ref{eq:S1_desired_est1}) follows from: 
\begin{equation}\label{eq:S1_desired_final}
    \sum_d d\bigg(\ssum{(m, q) = 1\\m\equiv 0(d)}\frac{\rho^\pm_m}{m}\bigg)^2\ll\frac{q}{\varphi(q)}\frac{1}{\log X}.
\end{equation}
The remainder of this section shall be dedicated to proving (\ref{eq:S1_desired_final}).
We begin with the statement of an analogue of~\cite[Lemma 6.1]{Ma}. 
\begin{lemma}\label{lem:general_est_S1}
    For any sequence of complex numbers $(\lambda_m)$ supported on $m$ coprime to $q$ and all of whose prime factors are $ < w$, we have that 
    \[
        \sum_{d}\bigg(\ssum{m\equiv 0(d)}\frac{\lambda_m}{m}\bigg)\ll \prod_{\substack{p < w\\p\nmid q}}\bigg(1 - \frac{1}{p}\bigg)
        \ssum{b\\ p | b\implies p < w\\ (b, q) = 1}\frac{b}{\varphi(b)^2}\ssum{e_1, e_2\\ (e_1, e_2) = 1\\ (e_1e_2, bq) = 1}
        \frac{|\theta_{be_1}||\theta_{be_2}|}{e_1e_2\varphi(e_1e_2)}
    \]
    where for $(b, q) = 1$, we take
    \[
        \theta_b = \sum_{d | b}\lambda_d.
    \]
\end{lemma}
\begin{proof}
    The proof of Lemma~\ref{lem:general_est_S1} is identical to that of~\cite[Lemma 6.1]{Ma}, 
    except throughout, sums are taken to be over numbers coprime to $q$. 
\end{proof}
Specializing to the case of $\lambda_d = \rho_d^\pm\charf{(d, q) = 1}$, we are done by using non-negativity to allow the sum to range over $b, e_1, e_2$,
possibly not coprime to $q$, taking 
\[
    \theta_b = \sum_{d|b}\rho_d^\pm.
\]
By the estimates at the end of~\cite[\S6.2]{Ma}, as our choice of $\beta$ satisfies $\pfrc{\beta}{\beta - 1}^{16}< 2$ for $\eps$ sufficiently small, the desired result follows.

\subsection{Bounding \texorpdfstring{$S_3$}{S3}}
We shall show in this section that 
\begin{equation}\label{eq:S3_ests}
    S_3^\alpha, S_3^\beta\ll\frac{\varphi(q)}{q}\frac{LX}{\log X} = \frac{\varphi(q)}{q}AX\varphi(q).
\end{equation}
This follows if we can show for $Y\in\set{3X, X^{10}}$ that 
\[
    \ssum{n\le Y\\ (n, q) = 1}\bigg(\sum_{d | n}\alpha_d^-\bigg)^2, \ssum{n\le Y\\ (n, q) = 1}\bigg(\sum_{d | n}\beta_d\bigg)^2\ll\frac{Y}{\log X}.
\]
Now, note that 
\begin{align*}
    \ssum{n\le Y\\ (n, q) = 1}\bigg(\sum_{d | n}\alpha_d^-\bigg)^2&\le\sum_{n\le Y}\bigg(\sum_{d | n}\alpha_d^-\bigg)^2\\
    &\ll\sum_{n\le Y}\bigg(\sum_{e | n}\rho_e^+\bigg)^2 + \sum_{n\le Y}\bigg(\sum_{e | n}\rho_e^-\bigg)^2\ll\frac{Y}{\log X}
\end{align*}
by~\cite[(71)]{Ma}. Also, by the definition of $\beta_d, \alpha_{d, P}^+$, we have that as in the reduction of (70) to (71) in~\cite{Ma},
\[
    \ssum{n\le Y\\ (n, q) = 1}\bigg(\sum_{d | n}\beta_d\bigg)^2
    \le\ssum{n\le Y}\bigg(\sum_{d | n}\beta_d\bigg)^2\le\sum_{n\le Y}\bigg(\sum_{e|n}\rho_e^+\bigg)^2
\]
The desired result then follows from~\cite[(71)]{Ma} again since our choice of $\beta$ satisfies $\pfrc{\beta}{\beta - 1}^{16} < 2$.

\subsection{Bounding \texorpdfstring{$S_2$}{S2}}
Our treatment shall be broadly similar to that of Lemma 4 of Mikawa~\cite{Mi}.
The well-factorability of linear sieve weights shall reduce the estimation of $S_2$ to Proposition~\ref{prop:disp_est}.
We use Poisson summation to reduce Proposition~\ref{prop:disp_est} to bounding a bilinear form with Kloosterman fractions whose coefficients are sieve weights.
An application of Cauchy-Schwarz to smooth out two of the coefficients, followed by the Weil bound for the resulting incomplete Kloosterman sum yields the result. 
\begin{proposition}\label{prop:disp_est}
    Suppose that $\alpha_1,\beta_1,\alpha_2,\beta_2 : \NN\to\CC$ are bounded. Suppose that $w$ is smooth and compactly supported in $\RR_{>0}$.
    Then, for $M_1, N_1, M_2, N_2\ge 1$, $a\ne 0, e, e'\ge 1, \delta | aq$ satisfying $(ee', q) = 1$, we have
    \begin{multline*}
       \ssum{m_1\sim M_1\\ n_1\sim N_1\\ m_2\sim M_2\\ n_2\sim N_2\\ (m_1n_1, m_2n_2) = \delta\\ (m_1n_1m_2n_2, q) = 1}\alpha_1(m_1)\beta_2(n_1)
       \alpha_2(m_2)\beta_2(n_2) [\dots]\\
        \ll\bigg(\frac{M_1\efr{1}{4}N_1\efr{1}{2}}{M_2\efr{1}{2}}
            + \frac{1}{M_1^{\frac{1}{4}}N_1^{\frac{1}{2}}}
        + \frac{X\efr{1}{2}}{M_1\efr{1}{2}M_2\efr{1}{2}N_1N_2}\bigg) (ee')^2
        M_1N_1M_2N_2(|a||q|)^{o(1)},
    \end{multline*}
    where 
    \[
        [\dots] = \bigg(\ssum{r\sim X\\ r\equiv -aq(em_1n_1)\\ r\equiv 0(e'm_2n_2)\\ (r, q) = 1}
        w\pfrc{r}{X} - \frac{1}{[em_1n_1, e'm_2n_2]}\frac{\varphi(q)}{q}\hat w(0)X\bigg).
    \]
\end{proposition}
First, we bound $S_2$ assuming Proposition~\ref{prop:disp_est}.
Note that by well-factorability properties of the linear sieve weights (see~\cite[\S12.7]{FI}), we may decompose $S_2$ into $O((\log X)^{O(1)})$-many sums of the form 
\begin{align*}
    \sum_{0 < |k|\le L/q}&\sum_{e, e'}\rho_e^\pm\rho_{e'}^\pm\int_{\max(0, qk)}^{L + \min(0, qk)}
    \sum_{d_1, d_2, d_3, d_4} 
    \alpha_1(d_1)\alpha_2(d_2)\alpha_3(d_3)\alpha_4(d_4)\\
    &\bigg(\ssum{m_1, m_2\\ ed_1d_2m_1 = e'd_2d_3m_2 + kq\\ (m_1, q) = 1} g\bigg(\frac{d_1m_1}{X} + \frac{y}{X}\bigg)
        - \frac{1}{[ed_1, e'd_2]}\cdot\frac{\varphi(q)}{q}\hat g(0)X\bigg)
\end{align*}
for bounded coefficients $\alpha_i$ supported on $d_i\sim D_i$ for scales $D_1, D_3\ll X^{\frac{1}{2} - 50\eps}$, $D_2, D_4\ll X^{\kappa - \frac{1}{2}}$.
Then, by Proposition~\ref{prop:disp_est} applied to the inner sum, this is at most
\begin{multline*}
    \frac{L}{q}\cdot L\cdot X^{2\kappa - 100\eps}\cdot\bigg(X^{\frac{1}{8} - \frac{25}{2}\eps + \frac{1}{2}(\kappa - \frac{1}{2})}
    + \frac{X\efr{1}{2}}{X^{\frac{1}{2} - 50\eps}X^{2\kappa -1}}\bigg) X^{o(1)}\\
    \ll\frac{L}{q}\left(X^{\frac{5}{2}(\kappa - \frac{1}{2}) - \frac{1}{2}} + X^{-50\eps}\right)LX.
\end{multline*}
For $c$ sufficiently small in terms of $\eps$ (recall that $L/q\le X^c$ in the statement of the main theorem), we obtain that
\begin{equation}\label{eq:S2_bd}
    |S_2|\ll LX^{1 - 40\eps}.
\end{equation}

It now remains to show Proposition~\ref{prop:disp_est}, which we do in the remainder of this section.
First, note by adjusting the choice of $\alpha_1,\alpha_2$, it suffices to show Proposition~\ref{prop:disp_est} in the case $e = e' = 1$, which we suppose from now on. 
Then, we have
\begin{multline*}
    \ssum{r\sim X\\ r\equiv -aq(m_1n_1)\\ r\equiv 0(m_2n_2)\\ (r, q) = 1}
    w\pfrc{r}{X} - \frac{1}{[m_1n_1, m_2n_2]}\frac{\varphi(q)}{q}\hat w(0)X\\
    = \sum_{d | q}\mu(d)\bigg(\ssum{r\sim X/d\\ r\equiv -aq\conj d(m_1n_1)\\ r\equiv 0(m_2n_2)}w\pfrc{r}{X/d} - \frac{1}{[m_1n_1, m_2n_2]}\hat w(0)\frac Xd\bigg). 
\end{multline*}
By the divisor bound and triangle inequality, it suffices to show that for any ${X_1\le X}$ and $d| q$, we have
\begin{equation}\label{eq:new_disp_est1}
    \mc S\ll\bigg(\frac{M_2\efr{1}{4}N_2\efr{1}{2}}{M_1\efr{1}{2}} + \frac{X\efr{1}{2}}{M_1\efr{1}{2}M_2\efr{1}{2}N_1N_2}\bigg) M_1N_1M_2N_2(|a||q|)^{o(1)}.
\end{equation}
Here $\mc S$ is equal to
\begin{align*}
    \ssum{m_1\sim M_1\\ n_1\sim N_1\\ m_2\sim M_2\\ n_2\sim N_2\\ (m_1n_1, m_2n_2) = \delta\\ (m_1n_1m_2n_2, q) = 1}&\alpha_1(m_1)\beta_2(n_1)
    \alpha_2(m_2)\beta_2(n_2) \\
    &\cdot\bigg(\ssum{r\sim X_1\\ r\equiv -aq\conj d(m_1n_1)\\ r\equiv 0(m_2n_2)}
    w\pfrc{r}{X_1} - \frac{1}{[m_1n_1, m_2n_2]}\hat w(0)X_1\bigg).
\end{align*}
Write $r = m_2n_2r'$, and let $\mu_i =(m_i, \delta), \nu_i = \delta/\mu_i$, $m_i' = \frac{m_i}{\mu_i}, n_i' = \frac{n_i}{\nu_i}, q_1 = aq/\delta$. Then, we have that 
\[
    r\equiv -aq\conj d(m_1n_1)\iff m_2'n_2'r'\equiv -q_1(m_1'n_1')\iff r'\equiv -\conj{m_2'n_2'}q_1 (\tilde m_1'n_1').
\]
Then, by Poisson summation, we have
\begin{align*}
    \mc S = \ssum{\mu_i\nu_i = \delta\\ (i\le 2)}\mc S_1(\mu_1,\nu_1,\mu_2,\nu_2),
\end{align*}
where $ \mc S_1(\mu_1, \nu_1,\mu_2,\nu_2)$ is
\begin{align*}
    &\ssum{m_i'\sim M_i' (i\le 2)\\ (m_i', \delta) = 1}\ssum{n_i'\sim N_i'\\ (n_i', \delta) = 1\\ (m_1'n_1', m_2'n_2') = 1}
    \alpha_1(m_1'\mu_1)\beta_1(n_1'\nu_1)\alpha_2(m_2'\mu_2)\beta_2(n_2'\nu_2)\\
    &\hspace{0.5cm}\cdot\frac{m_1'n_1'm_2'n_2'}{\delta X_1}\sum_{\ell\ne 0}\hat w\pfrc{\ell}{\delta X_1/(m_1'n_1'm_2'n_2')} 
    e\pfrc{q_1\ell\conj{dm_2'n_2'}}{\tilde m_1'n_1'}.
\end{align*}
Now, fix some choice of $\mu_1, \nu_1,\mu_2,\nu_2$. It suffices to show that 
\[
    \mc S_1\ll \bigg(\frac{M_2\efr{1}{4}N_2\efr{1}{2}}{M_1\efr{1}{2}} + \frac{X\efr{1}{2}}{M_1\efr{1}{2}M_2\efr{1}{2}N_1N_2}\bigg) M_1N_1M_2N_2(|a||q|)^{o(1)}.
\]
Take $\eps_1 > 0$ small, and let $L_0 = X^{\eps_1}\frac{\tilde e\tilde M_1'N_1'M_2'N_2'}{\delta X_1}$.
Because $w$ is Schwartz, the contribution of $|\ell| > L_0$ is $O(X^{-2023})$. 
Therefore, it suffices to bound the contribution $\mc S_1'$ of $|\ell|\le L_0$. 
By partial summation, for some choice of $L\le L_0$, $\mc S_1'$ is 
\begin{align*}
    \ll X^{\eps_1}\ssum{m_i'\sim M_i' (i\le 2)\\ (m_i', \delta) = 1}\ssum{n_i'\sim N_i'\\ (n_i', \delta) = 1\\ (m_1'n_1', m_2'n_2') = 1}
    &\alpha_1(m_1'\mu_1)\beta_1(n_1'\nu_1)\alpha_2(m_2'\mu_2)\beta_2(n_2'\nu_2)\\
    &\cdot\frac{1}{L_0}\sum_{0 < |\ell|\le L} e\pfrc{q_1\ell\conj{\tilde dm_2'n_2'}}{\tilde m_1'n_1'}.
\end{align*}
By Cauchy-Schwarz, we have that 
\[
    |\mc S_1'|^2\ll X^{\eps_1}M_1'M_2'\Omega,
\]
where
\[
    \Omega = \ssum{m_i'\sim M_i' (i\le 2)\\ (m_i', \delta) = 1}
    \bigg|\frac{1}{L_0}\sum_{0 < |\ell|\le L}\ssum{n_i'\sim N_i'(i\le 2)\\ (m_1'n_1', m_2'n_2') = 1}\beta_1(n_1'\nu_1)\beta_2(n_2'\nu_2)
    e\pfrc{q_1\ell\conj{\tilde dm_2'n_2'}}{\tilde m_1'n_1'}\bigg|^2.
\]
Expanding and applying the triangle inequality yields that 
\begin{align*}\label{eq:cs_expand}
    \Omega\le\sum_{m_1'\sim M_1'}&\frac{1}{L_0^2}\sum_{0 < |\ell_1|, |\ell_2|\le L}\ssum{n_i', \tilde n_i'\sim N_i'(i\le 2)\\ 
    (m_1'n_1', n_2') = 1\\ (m_1'\tilde n_1', \tilde n_2') = 1}\\ 
    &\bigg|\ssum{m_2'\sim M_2'\\ (m_2', m_1'n_1'\tilde n_1') = 1} 
    e\bigg(\frac{q_1\ell_1\conj{dm_2'n_2'}}{\tilde m_1'n_1'} - \frac{q_1\ell_2\conj{dm_2'\tilde n_2'}}{\tilde m_1'\tilde n_1'}\bigg)\bigg|.\numberthis
\end{align*}
It can be checked that 
\[
    \frac{\ell_1\conj{dm_2'n_2'}}{m_1'n_1'} - \frac{\ell_2\conj{dm_2'\tilde n_2'}}{m_1'\tilde n_1'}
    = \Delta\frac{\conj{dm_2'n_2'\tilde n_2'}}{m_1'\frac{n_1'}{(n_1', \tilde n_2')}\frac{\tilde n_1'}{(\tilde n_1', n_2')}}
\]
where 
\[
    \Delta = \ell_1\frac{\tilde n_2'}{(n_1', \tilde n_2')}\cdot\frac{\tilde n_1'}{(\tilde n_1', n_2')} 
    - \ell_2\frac{n_1'}{(n_1', \tilde n_2)}\cdot\frac{n_2'}{(\tilde n_1', n_2')}.
\]
Let $\mc O$ be the contribution of $\Delta\ne 0$ and $\mc D$ be that of $\Delta = 0$. We have that by the divisor bound
\begin{equation}\label{eq:diagonal_bd}
    \mc D\ll \frac{1}{L_0}M_1'M_2'N_1'N_2'X^{o(1)}
\end{equation}
Then, by Lemma~\ref{lem:kloost_est} together with (\ref{eq:cs_expand}), we have that 
\begin{align*}\label{eq:post_weil_gcds}
    \mc O\ll X^{o(1)}&(M_1'N_1'^2)^{\frac{1}{2}}\bigg(1 + \frac{M_2'}{M_1'N_1'^2}\bigg)\\
    &\cdot\ssum{m_1'\sim M_1'\\ (m_1', \delta) = 1}
    \frac{1}{L_0^2}\sum_{0 < |\ell_1|, |\ell_2|\le L}\ssum{n_i', \tilde n_i'\sim N_i'(i\le 2)\\ 
    (m_1'n_1', n_2') = (m_1'\tilde n_1', \tilde n_2') = 1\\ \Delta\ne 0} 
    (q_1\Delta, m_1'n_1'\tilde n_1')^{\frac{1}{2}}.\numberthis
\end{align*}
If we write 
\[
    u = \ell_1\tilde n_1'\tilde n_2', 
    h = \Delta(n_1', \tilde n_2')(\tilde n_1', n_2') 
    = \ell_1\tilde n_1'\tilde n_2' - \ell_2n_1'n_2', 
\]
then by the divisor bound, the right-hand side of (\ref{eq:post_weil_gcds}) is 
\begin{align*}
    &\ll X^{o(1)}(M_1'N_1'^2)^{\frac{1}{2}}\bigg(1 + \frac{M_2'}{M_1'N_1'^2}\bigg)\\
    &\hspace{1cm}\frac{1}{L_0^2}\sum_{u\le 10L_0N_1'N_2'}\ssum{0 < |h|\le 20L_0N_1'N_2'\\ u + h\ne 0}
    \sum_{m_1'\sim M_1'} (q_1h, u(u + h))\efr{1}{2} (q_1h, m_1')\efr{1}{2}\\
    &\ll X^{o(1)}(M_1'N_1'^2)^{\frac{1}{2}}\bigg(1 + \frac{M_2'}{M_1'N_1'^2}\bigg)\\
    &\hspace{1cm}\frac{1}{L_0^2}\ssum{0 < |h|\le 20L_0N_1'N_2'}\sum_{m_1'\sim M_1'} (q_1h, m_1')\efr{1}{2}
    \sum_{u\le 30L_0N_1'N_2'} (q_1h, u).
\end{align*}
By several applications of Lemma~\ref{lem:gcd_sum}, we obtain that
\begin{equation}\label{eq:off_diag_bd}
    \ll (M_1'N_1'^2)^{\frac{1}{2}}\bigg(1 + \frac{M_2'}{M_1'N_1'^2}\bigg)
    M_1'N_1'^2N_2'^2 (q_1X)^{o(1)}.
\end{equation}
The desired result follows upon combining (\ref{eq:diagonal_bd}), (\ref{eq:off_diag_bd}).

\subsection{Collecting bounds}
Combining (\ref{eq:S1_bd}), (\ref{eq:S2_bd}), and (\ref{eq:S3_ests}), we obtain 
(\ref{eq:remainder_averaged}).

\section{Acknowlegements}
The author would like to thank J. Maynard, K. Matom\"aki, and J. Ter\"av\"ainen for discussions of this problem. The author would also like to thank J. Ter\"av\"ainen and P. Sarnak for comments on earlier drafts
of this paper.

\end{document}